\documentclass[a4paper,12pt]{amsart}

\usepackage{amsmath}
\usepackage{amssymb}
\usepackage{mathrsfs}
\usepackage{enumerate}
\usepackage{ifthen}
\usepackage{graphicx}
\usepackage[T1]{fontenc} %skandit

\nonstopmode \numberwithin{equation}{section}
\newtheorem{theorem}{Theorem}[section]
\newtheorem{corollary}{Corollary}[section]
\newtheorem{lemma}{Lemma}[section]

\theoremstyle{definition}
\newtheorem{definition}{Definition}[section]

\newtheorem{problem}{Problem}[section]

\begin{document}

\title{Notes on group distance magicness of product graphs}

\author{A V Prajeesh}
\address{A V Prajeesh, Department of Mathematics, National Institute of Technology Calicut, Kozhikode~\textnormal{673601}, India. }
\email{prajeesh\_p150078ma@nitc.ac.in}

\author{K Paramasivam}
\address{Krishnan Paramasivam, Department of Mathematics, National Institute of Technology Calicut,  Kozhikode~\textnormal{673601}, India. }
\email{sivam@nitc.ac.in}


%\def\thefootnote{}
%\footnotetext{ {\tiny File:~\jobname.tex,
%printed: \number\year-\number\month-\number\day,
%          \thehours.\ifnum\theminutes<10{0}\fi\theminutes }
%} \makeatletter\def\thefootnote{\@arabic\c@footnote}\makeatother

\begin{abstract}
	If $l$ is a bijection from the vertex set $V(G)$ of a graph $G$ to an additive abelian group $\Gamma$ with $|V(G)|$ elements in such a way that for any vertex $u$ of $G$, the weight $w_{G}(u) = \sum_{v\in N_{G}(u)}l(v)$ is equal to the same element $\mu_o$ of $\Gamma$, then $l$ is called a $\Gamma$-distance magic labeling of $G$. A graph $G$ that admits such a labeling is called $\Gamma$-distance magic and if $G$ is $\Gamma$-distance magic for every additive abelian group $\Gamma$ with $|V(G)|$ elements, then $G$ is called a group distance magic graph. In this paper, we provide few results on the group distance magic labeling of product graphs, namely lexicographic and direct product of two graphs. We also prove some necessary conditions for a graph to be group distance magic and provide a characterization for a tree to be group distance magic.
\end{abstract}

\thanks{}

\maketitle
\pagestyle{myheadings}
\markboth{A V Prajeesh and Krishnan Paramasivam }{Notes on group distance magicness of product graphs}

\section{Introduction}
\noindent In this paper, we consider only simple and finite graphs. We use $V(G)$ for the vertex set and $E(G)$ for the edge set of a graph $G$. The neighborhood $N_G(v)$, or shortly $N(v)$ of a vertex $v$ of $G$ is the set of all vertices adjacent to $v$, and the degree $deg_G(v)$, or shortly $deg(v)$ of $v$ is the number of vertices in $N_G(v)$. The distance $d_G(u, v)$ between two vertices $u$ and $v$ of $G$ is the length of shortest path connecting $u$ and $v$. For more standard graph theoretic notation and terminology, we refer Bondy and
Murty \cite{bondy2008graph} and Hammack $et$ $al.$ \cite{hammack2011handbook}.
\par Recall two standard graph products (see \cite{hammack2011handbook}). Let $G$ and $H$ be two graphs. Both, the lexicographic product $G \circ H$ and the direct product $G\times H$ are graphs with the vertex set $V(G)\times V(H)$. Two vertices $(g,h)$ and $(g',h')$ are adjacent in:
\begin{itemize}
	\item[(i)] $G\circ H$ if and only if $g$ is adjacent to $g'$ in $G$, or $g=g'$ and $h$ is adjacent to $h'$ in $H$.
	\item[(ii)] $G\times H$ if and only if $g$ is adjacent to $g'$ in $G$ and $h$ is adjacent to $h'$ in $H$.
\end{itemize}
\par A \textit{distance magic labeling} of $G$ is a bijection $l : V(G) \rightarrow \{1,...,|V(G)|\}$, such that for any $u$ of $G$, the weight of $u$, $w_{G}(u) = \sum\limits_{v\in N_{G}(u)}l(v) $ is a constant $\mu$. A graph $G$ that admits such a labeling is called a distance magic graph \cite{vilfred}.
\par The concept of distance magic labeling was studied by Vilfred \cite{vilfred} as sigma labeling. Later, Miller $et$ $al.$ \cite{miller2003distance} called it a $1$-vertex magic vertex labeling and Sugeng $et$ $al.$ \cite{sugeng2009distance} referred the same as distance magic labeling.
\par The concept of distance magic labeling has been motivated by the construction of magic squares.  It is worth to mention the motivation given by Froncek $et$ $al.$ \cite{froncek2006fair} through an equalized incomplete tournament. An equalized incomplete tournament of $n$ teams with $r$ rounds,
$EIT(n, r)$ is a tournament which satisfies the following conditions:
\begin{itemize}
	\item [(i)] every team plays against exactly $r$ opponents.
	\item [(ii)] the total strength of the opponents, against which each team plays is a constant.
\end{itemize}
Therefore, finding a solution to an $EIT(n,r)$ is equivalent to obtain a distance magic labeling of an $r$-regular graph with $n$ vertices.
%\par The concept of distance magic labeling was studied by Vilfred \cite{vilfred} as sigma labeling. Miller $et$ $al.$ \cite{miller2003distance} call it a $1$-vertex magic vertex labeling and Sugeng $et$ $al.$ \cite{sugeng2009distance} refer the same as distance magic labeling.\\
\par Most of important results and problems, which are more relevant and helpful in proving our results, are listed below.
\begin{theorem}
	\textnormal{\cite{vilfred,miller2003distance,jinnah,rao}} No $r$-regular graph with $r$-odd can be a distance magic graph.
\end{theorem}
\begin{lemma}
	\textnormal{\cite{miller2003distance}} If $G$ contains two vertices $u$ and $v$ such that $|N(u) \cap N(v)| = deg(v)-1 = deg(u)-1$, then $G$ is not distance magic.
\end{lemma}
\begin{theorem}\label{shafiq}
	\textnormal{\cite{shafiq2009distance}} Let $r \geq 1$ and $n \geq 3$. If $G$ is an $r$-regular graph and $C_{n}$ the cycle of
	length $n$, then $G\circ C_{n}$ admits a labeling if and only if $n = 4$.
\end{theorem}
\par In 2009, Shafiq $et$ $al.$\cite{shafiq2009distance} posted a problem of the existence of distance magic labeling of the lexicographic product of a non-regular graph $G$ with $C_{4}$.
\begin{problem}\label{prb1}
	\textnormal{\cite{shafiq2009distance}} If $G$ is a non-regular graph, determine if there is a distance magic labeling of $G \circ C_{4}$.
\end{problem}
\par In 2018, Cichacz and G{\"o}rlich \cite{cichacz2018constant}, raised a similar question in the case of direct product of $G$ with $C_{4}$.
\begin{problem}\label{prb2}
	\textnormal{\cite{cichacz2018constant}} If $G$ is non-regular graph, determine if there is a distance magic labeling of $G \times C_{4}$.
\end{problem}
\par Anholcer $et$ $al.$ \cite{anholcer2015distance} defined a distance magic graph $G$ to be \textit{balanced} if there exists a bijection $l:V(G) \rightarrow \{1,...,|V(G)|\}$ such that for any $w$ of $G$, the following holds: \\ if $u \in N(w)$ with $l(u)=i$, then $\exists$ $v \in N(w), v \neq u ,$ with $l(v)= |V(G)|+1-i$. \\ Further, we call $u$, the twin vertex of $v$ and vice versa.
\par From \cite{anholcer2015distance}, It is clear that $G$ is a balanced distance magic graph or shortly, balanced-$dmg$ if and only if $G$ is regular and the vertex set of $G$ can be expressed as $\{v_{i},v_{i}': 1\leq i\leq \frac{|V(G)|}{2}\}$ such that  for any $i$, $N(v_{i}) = N(v_{i}')$, where $v_i$ is the twin vertex of $v_i'$. The graphs $K_{2n,2n}$ and $K_{2n}-M$, $M$ any perfect matching of $K_{2n}$ are examples of balanced-$dmg$'s.
\par The $k^{\textnormal{th}}$ power of a graph $G$ is a graph $G^{k}$ with the same set of vertices as $G$ and any two vertices $u$ and $v$ are connected if and only if $d_G(u,v)\leq k$.
\par In 2016, Arumugam $et$ $al.$ \cite{arumugam2016distance} proved the following result for the characterization of entire class of distance magic graphs $G$ with $\Delta(G) = |V(G)|-1.$ It is interesting to see that the following class of graphs is derived from balanced-$dmg$.
\begin{theorem}
	\textnormal{\cite{arumugam2016distance}}
	Let $G$ be any graph of order $n$ with $\Delta(G)= n-1.$ Then $G$ is a distance magic graph if and only if $n$ is odd and $G \cong (K_{n-1}-M) + K_{1}$ where $M$ is a perfect matching in $K_{n-1}.$
\end{theorem}
\par In 2004, Rao \cite{rao} proved the following result.
\begin{theorem}
	\textnormal{\cite{rao}}
	The graph $C_{k}\Box C_{m}$ is distance magic if and only if $k=m\equiv 2\mod 4.$
\end{theorem}
\par Now, a natural question arises that for all graphs, which are not distance magic, whether one can introduce a new concept by replacing the existing co-domain, $\{1,..., |V(G)|\}$ of the distance magic labeling $l$ by an another set or even by an algebraic structure such as group so that these graphs can admit such a magic-type labeling. 
\par Motivated by this fact, in 2013, Froncek \cite{froncek2013group}, introduced the notion of group distance magic labeling of graphs. He proved that $C_{k}\Box C_{m}$ is $\mathbb{Z}_{km}$-distance magic if and only if $km$ is even, where $\mathbb{Z}_{km}$ is a finite additive abelian group.
% Through out this paper, the algebraic structure $\Gamma = (\Gamma,+)$ is a finite abelian group, together with associative binary operation addition $+$. The order and identity element of $\Gamma$ are denoted by $|\Gamma|$ and $e$ respectively. \\\\
\par Throughout this paper, the algebraic structure $\Gamma=(\Gamma, +)$ is a finite additive abelian group or shortly, abelian group, where $\lq +$' is a binary operation on $\Gamma $. The order and identity element of $\Gamma$ are denoted by $|\Gamma|$ and $e$ respectively. Recall that any non-identity element $g$ of $\Gamma$ is an involution if $g=-g$, where $-g$ is the additive inverse of $g$. If $g$ is an involution of $\Gamma$, then $2g=e$. Also, any non-trivial finite group $\Gamma$ has involution if and only if $|\Gamma|$ is even. The fundamental theorem of abelian group states that a finite abelian group $\Gamma$ can be expressed as a direct product of cyclic groups of prime power order, where the product is unique up to the order of subgroups. Moreover, the sum of all elements of $\Gamma$, $s(\Gamma) = \sum _{g\in \Gamma}g$ is equal to the sum of all involutions of $\Gamma$.
\begin{lemma}\label{grp1}
	\textnormal{\cite{combe2004magic}} Let $\Gamma$ be an abelian group.
	\begin{enumerate}
		\item[\textnormal{(i)}] If~ $\Gamma$ has exactly one involution $g'$, then $s(\Gamma)=g'$.
		\item[\textnormal{(ii)}] If~ $\Gamma$ has no involutions, or more than one involution, then $s(\Gamma)=e$.
	\end{enumerate}
\end{lemma}
\par For more group theory related terminology and notation, we refer Herstein \cite{herstein2006topics}.
\begin{definition}
	\par\textnormal{\cite{froncek2013group}} If $\Gamma$ is an abelian group and $G$ is a graph such that $|V(G)|=|\Gamma|$, then a bijection $l:V(G)\rightarrow \Gamma$ is said to be a \textit{$\Gamma$-distance magic labeling} of $G$ if for any $u$ of $G$, the weight of $u$, $w_{G}(u) = \sum_{v\in N_{G}(u)}l(v)$ is equal to the same element $\mu_{o}$ of $\Gamma$. A graph $G$ that admits such a labeling is called a $\Gamma$-distance magic graph and the element $\mu_{o}$ is called the magic constant associated with the labeling $l$ of $G$.
\end{definition} % A $\Gamma$-distance magic labeling of a graph $G$ on $n$ vertices is a bijection $l$ from $V$ to $\Gamma$ with $|\Gamma|=n,$ such that for any vertex  $u\in V$, the weight $w_{G}(u) = \sum_{v\in N_{G}(u)}l(v)$ is equal to the same element $\mu \in \Gamma$, called the magic constant.\\
\par Whenever $l$ is a distance magic labeling of a graph $G$ on $n$ vertices with the magic constant $\mu$, then consider a new labeling $l^*$ on $G$ as,

\[l^*(v)= \begin{cases} 
l(v) & \textnormal{~if~~} l(v)<n \\
0 & \textnormal{~if~~} l(v)=n,
\end{cases}
\]
which is a $\mathbb{Z}_n$-distance magic labeling with magic constant $\mu_{0}$, where $ \mu_{0}\equiv\mu~(\textnormal{mod}~n)$. On the other hand, it is observed from \cite{froncek2013group} that  $G$ is a $\mathbb{Z}_{n}$-distance magic does not imply that $G$ is distance magic.
\par In 2014, Cichacz \cite{cichacz2014distance,cichacz2014note} proved the following results.
\begin{theorem}\label{cichacz1}
	\textnormal{\cite{cichacz2014note}}	Let $G$ be a graph of order $n$ and $\Gamma$ be an arbitrary abelian group of order $4n$ such that $\Gamma\cong \mathbb{Z}_2\times \mathbb{Z}_2\times \mathcal{B}$ for some abelian group $\mathcal{B}$ on $n$ vertices. Then there exists a $\Gamma$-distance magic labeling for the graph $G\circ C_{4}$.
\end{theorem}
\begin{theorem}
	\textnormal{\cite{cichacz2014note}}	Let $G$ be a graph of order $n$ and $\Gamma$ be an abelian group of order $4n$. If $n = 2^p(2k+1)$ for some natural numbers $p,k$ and $deg(v) \equiv c \mod 2^{p+1}$ for some constant $c$ for any $v\in V(G)$, then there exists a $\Gamma$-distance magic labeling for the graph $G\circ C_4$.
\end{theorem}
\begin{theorem}
	\textnormal{\cite{cichacz2014distance}}	Let $G$ be a graph of order $n$ and if $n = 2^p(2k+1)$ for some natural numbers $p,k$ and $deg(v) \equiv c \mod 2^{p+2}$ for some constant $c$ for any $v\in V(G)$, then there exists a $\Gamma$-distance magic labeling for the graph $G\times C_4$.
\end{theorem}
\par Cichacz in \textnormal{\cite{cichacz2014notebi}} gave a complete characterization of group distance magicness of the complete bipartite graph.
\begin{theorem}\textnormal{\cite{cichacz2014notebi}}\label{bipart}
	The complete bipartite graph $K_{m,n}$ is a group distance magic graph if and only if $m + n \not\equiv2 \mod 4.$
\end{theorem}
\par Recently, Anholcer $et$ $al.$\cite{anholcer2014group} discussed the group distance magicness of the direct product of two graphs and obtained the following result.
\begin{theorem}\label{directproduct}
	\textnormal{\cite{anholcer2014group}}
	If $G$ is a balanced distance magic graph and $H$ an $r$-regular graph for $r\geq1$, then $G \times H$ is a group distance magic graph.
\end{theorem}
\par In the following section, we provide some necessary conditions for a graph to be group distance magic and also characterize the group distance magic labeling of a tree $T$. Further, we also discuss the group distance magic labeling of bi-regular graphs $G$ with regularities $r_{1}=|V(G)|-1$ and $r_{2}= 1,2,3,|V(G)|-2$ or $|V(G)|-3$.
\par Motivated by the results from \cite{cichacz2014distance,cichacz2014note,anholcer2014group}, in third section, we discuss the group distance magic labeling of $G\circ H$ and $G\times H$, where $G$ is a non-regular graph and $H$ is a balanced-$dmg$.
\section{Group distance magic labeling of $(K_{n-1}-M)+K_{1}$ and trees}

\begin{lemma}
	If $\Gamma$ is any abelian group and if $G$ is a graph with $|\Gamma|$ vertices such that $G$ has at least two distinct vertices of degree $|\Gamma|-1$, then $G$ is not $\Gamma$-distance magic.
\end{lemma}
\begin{proof}
	On the contrary, let $l$ be a $\Gamma$-distance magic labeling of the graph $G$ with magic constant $\mu_{0} $.  Let $u$ and $v$ be two distinct vertices of $G$ such that  $deg_{G}(u)=|\Gamma|-1=deg_{G}(v)$. Then,
	\begin{equation*}
	s(\Gamma) = \mu_{0}+l(u) = \mu_{0}+l(v).
	\end{equation*}
	\par Since the left cancellation law holds in $\Gamma$, we have $l(u)=l(v)$, a contradiction.
\end{proof}
\begin{lemma}\label{nbd2}
	If $\Gamma$ is any abelian group and if $G$ is a graph with $|\Gamma|$ vertices such that $G$ has two distinct vertices $u$ and $v$ with $|N_G(u)\cap N_G(v)|= deg_G(u)-1=deg_G(v)-1$, then $G$ is not $\Gamma$-distance magic.
\end{lemma}
\begin{proof}
	On the contrary, let $l$ be a $\Gamma$-distance magic labeling of the graph $G$ with magic constant $\mu_{0} $. Choose two vertices $u'$ and $v'$ of $G$ in such a way that the followings hold:
	\begin{enumerate}
		\item[(i)] $u'\in N_G(u)$ and $u'\not\in N_G(v)$, and
		\item[(ii)] $v'\in N_G(v)$ and $v'\not\in N_G(u)$.
	\end{enumerate}
	\par By comparing the weights of $u$ and $v$, we have
	\begin{equation*}
	\mu_{0} = w_{G}(u)= g+l(u') = g+l(v') = w_{G}(v),
	\end{equation*} for some $g$ in $\Gamma.$  By left cancellation law in $\Gamma$, $l(u') = l(v')$, a contradiction.
\end{proof}
\par The following result characterizes the group distance magicness of a tree.
\begin{theorem}\label{treegrp}
	A non-trivial tree $T$ is $\Gamma$-distance magic for an abelian group $\Gamma$ if and only if $T\cong K_{1,n}$, with $n\not\equiv 1\mod 4$.
\end{theorem}
\begin{proof}
	If $diam(T)=2$, the result is straightforward by Theorem \ref{bipart}. On the other hand if $diam(T)>2$, then $T$ has two vertices $u$ and $v$ such that $d_{G}(u,v)=diam(T)$. Hence the result follows from Lemma \ref{nbd2}.
\end{proof}
\begin{theorem}
	Let $\Gamma$ be any abelian group having atleast one element $g^\dagger$ such that $g^\dagger$ is not an involution. If~~$ l_{1}$ is a $\Gamma$-distance magic labeling of $G$ with magic constant $\mu_{0}$, then there exists a $\Gamma$-distance magic labeling $l_{2}$ of $G$ with magic constant $-\mu_{0}$.
\end{theorem}
\begin{proof}
	Let $l_1$ be the $\Gamma$-distance magic labeling of a graph $G$ with magic constant $\mu_{0}.$ Let $u$ be any vertex of $G$ with the neighbors $v_1, v_2,...,v_t$, where $t=|N_G(u)|$. Then,
	$$\displaystyle w_{G}(u) = \sum_{i=1}^{t}l_{1}(v_{i}) =\mu_{0}.$$
	\par Consider a new function $l_{2}(u) =-(l_{1}(u))$. Since $l_{1}$ is a bijection, $l_{2}$ is also a bijection. Further, $l_{1}$ is not identically equal to $l_{2}$ because if $l_{1}(v)=l_{2}(v)$ for all $v$, in particular, if $l_{1}(v^\dagger)  = l_{2}(v^\dagger) = g^\dagger$, then $l_{1}(v^\dagger)  = -l_{1}(v^\dagger) \implies 2l_{1}(v^\dagger) = e$, that is $2g^\dagger = e$, a contradiction.
	\par  In an abelian group, the inverse of sum of elements is equal to the sum of inverses of each elements. Thus, the weight of $u$ of $G$ with respect to $l_{2}$ is given by $$\displaystyle \sum_{i=1}^{t}l_{2}(v_{i})=\sum_{i=1}^{t}-l_{1}(v_{i})=-\biggl(\sum_{i=1}^{t}l_{1}(v_{i})\biggr) =-\mu_0.$$
	Since the choice of $u$ of $G$ is arbitrary, the result follows.
\end{proof}
\begin{theorem}
	Let $n>1$ be an odd integer. Let $G$ be a graph isomorphic to $(K_{n-1}-M)+K_1 $, where $M$ is any perfect matching of $K_{n-1}$. If $\Gamma$ is any abelian group with $|\Gamma|=n$, then $G$ admits a $\Gamma$-distance magic labeling $l$ if and only if $l(v_0)= e$, where $G[\{v_0\}] = K_1$.
\end{theorem}
\begin{proof}
	Let $\Gamma$ be an abelian group with $|\Gamma|=n$.
	Define the vertex set of $G$ as, $\{v_0\} \cup \{v_{i},v_{i}': 1\leq i\leq \frac{n-1}{2}\}$, where
	$v_{i}$ and $v_{i}'$ are twin vertices of $K_{n-1}-M$ and $v_{0}$ is the vertex, which induces $K_{1}$ of $G$. Let $l$ be a $\Gamma$-distance magic labeling of $G$ with magic constant $\mu_{0}$. We know that $ng=e$ for any element $g$ of $\Gamma$, in particular, $n\mu_{0} = e$. By using Lemma \ref{grp1} and by comparing the total weights, we get
	\begin{equation*}
	n\mu_{0} = (n-1)l(v_{0})+(n-2)\bigl(s(\Gamma)-l(v_{0})\bigr),
	\end{equation*}
	which implies that $l(v_0)$ is $e$.
	\par When $n$ is odd, for every element $g$ of $\Gamma$, there exists a unique $-g$ different from $g$ in $\Gamma$ such that $g+(-g) = e$.  Consider a function $l$ from $V(G)$ to $\Gamma$ as, $l(v_{0})=e$ and if $l(v_{i}) = g$, then $l(v_{i}')= -g,$ for  all $i\ne0$. It is not hard to verify that, the weight of each vertex of $G$  is $e$. Thus, $l$ is a $\Gamma$-distance magic labeling of $G$ with magic constant $\mu_{0}=e$.
\end{proof}
%\begin{theorem}
%	Let $G \cong (K_{n-1}-M)+K_1 $ be a graph on $n>1$ vertices where $n$ is odd %and  $M$ is any perfect matching in $K_{n-1}$, then $G \circ C_4 $ is  not distance magic.
%\end{theorem}
%It is interesting to see that the graph, $G\circ C_{4}$, where $G \cong (K_{n-1}-M)+K_{1} $ is not distance magic but it is $\Gamma$-distance magic for an abelian group  $\Gamma\cong \mathbb{Z}_{2}\times \mathbb{Z}_{2}\times \mathcal{B},$ where $\mathcal{B}$ is an abelian group with $|\mathcal{B}|=n$.
%The above Theorem and Lemma 2.2 helps us to state the following result.
\par Theorem \ref{bipart} and \ref{treegrp} confirm the fact that the number of group distance magic graphs on $n$ vertices with maximum degree $n-1$ is comparatively higher than the number of distance magic graphs with maximum degree $n-1$. Hence it is worth mentioning the group distance magic labeling of bi-regular graph with regularities $r_{1}= n-1$ and $r_{2}$, where $r_{2}$ is from the set $\{1,2,3,n-3\}$.
\par Let $n>3$ be an odd integer and $G$ be a bi-regular graph on $n$ vertices with a unique vertex $v$ of degree, $r_{1}=n-1$, and all other vertices of degree $r_{2}$, where $r_{2}=1,2,3$ or $n-3$. It is observed that if $l(v) \not= e$, then $l$ is not an $\Gamma$-distance magic labeling of $G$, for any abelian group $\Gamma$ with $|\Gamma|=n$.
\par Let $v_{e}$ and $v_{-g}$ be the vertices of $G$ labeled with $e$ and $-g$ respectively. On contrary, if for a given abelian group $\Gamma$ with  $|\Gamma|=n$, there exists a $\Gamma$-distance magic labeling $l$ such that $l(v)=g\not=e$. Then, $w_{G}(v)=-g$, for every $v$ in $G$.\\
\textbf{Case 1}: If $r_{2}=1$, then comparing $w_{G}(v_{e})$ and $w_{G}(v)$, we get $g=-g\implies 2g=e$, a contradiction.\\
\textbf{Case 2}: If $r_{2}=2$, then there exists a unique vertex $v^{*}\neq v$, $v_{e}\in N_{G}(v^{*})$. Now, comparing the weights $w_{G}(v^{*})$ and $w_{G}(v)$, we get $2g=e$, a contradiction.\\
\textbf{Case 3}: If $r_{2}=3$,	then consider a vertex $v^{*}\neq v$, $v^{*}\in N_{G}(v_{-g})$. Further, comparing the weights $w_{G}(v^{*})= g+(-g)+l(v^{**}) = w_{G}(v)=-g$, we get $l(v^{**})=-g $, a contradiction for $l$ being one-one.\\
\textbf{Case 4}: If $r_{2}=n-3$, then we get $e = n(-g) = (n-1-(n-3))g$ or $2g=e$, a contradiction.
\par When $G$ is a regular graph and $H$ is a balanced-$dmg$, the distance magicness and the group distance magicness of $G\circ H$ $(  $only when $ H\cong C_{4})$ and $G\times H$ are characterized by Theorem \ref{shafiq} and \ref{directproduct} respectively. In the case of a non-regular graph $G$, analogous to Problem \ref{prb1} and \ref{prb2}, natural questions arise on the existence of group distance magic labeling of $G\circ H$ and $G\times H$. The following section provides partial solutions to these problems.
\section{Group distance magic labeling of lexicographic product and direct product of two graphs}
\noindent Throughout this section, we assume that $H$ is a balanced-$dmg$ on either $2^{k}$ or $4k+2$ vertices, $\Gamma$ is an abelian group and $\mathcal{A}$ is an abelian group with elements, $a_{0},a_{1}...,a_{|\mathcal{A}|-1}$, where $a_0$ is the identity element in $\mathcal{A}$. % Also note that the vertex set of $G\circ H$ and the vertex set $G\times H$ is the Cartesian product $V(G)\times V(H)$.\\
Observe that $C_{4}\cong K_{4}-M$ is balanced-$dmg$ of order $2^2$ and $C_{4k+2}^{2k}\cong K_{4k+2}-M$ is a balanced-$dmg$ of order $4k+2$, where $M$ is a perfect matching.
%Therefore, in some sense the results in this section generalize the results by Cichacz \cite{cichacz2014distance,cichacz2014note}.
\begin{theorem}
	Let $G$ be a graph on $n$ vertices and $\Gamma$ be an abelian group with $|\Gamma|=(4k+2)n$ such that $\Gamma\cong \mathbb{Z}_{4k+2}\times \mathcal{A}$, where $k\geq 1$ and  $\mathcal{A}$ an abelian group with $|\mathcal{A}|=n$.
	\begin{enumerate}
		\item[\textnormal{(i)}] If the degree of the vertices of $G$ are either all even or all odd, then $G\circ C_{4k+2}^{2k}$ is $\Gamma$-distance magic.
		\item[\textnormal{(ii)}] If there exists a constant $m\in \mathbb{N}$ such that $deg_{G}(u) \equiv m\mod4k+2$, for all $u\in V(G)$, then $G \times C_{4k+2}^{2k}$ is $\Gamma$-distance magic.
	\end{enumerate}
\end{theorem}
\begin{proof}
	Let $G$ be a graph with the vertices $u_{0},...,u_{n-1}$ and $H \cong C_{4k+2}^{2k}$ be a balanced-$dmg$ with the vertices
	$x^{0},x^{0'},...,x^{2k},x^{(2k)'}$.
	% For each $i=0$ to $n-1$, let $x_{i}^{j}$ and $x_{i}^{j'}$ be the vertices of $G \circ H$ and $G\times H$ that replace $u_i$ of $G$. \\
	For any $i\in\{0,...,n-1\},$ let $H_{i}=\{x_{i}^{0}, x_{i}^{0'},...,x_{i}^{2k}, x_{i}^{(2k)'}\}$ be the vertices of $G \circ H$ and $G\times H$ that replace $u_i$ of $G$.
	\par Using the isomorphism $\phi:\Gamma \rightarrow \mathbb{Z}_{4k+2}\times \mathcal{A}$, we identify $g\in \Gamma$ with its image $\phi(g) = (z,a_{i})$, where $z\in \mathbb{Z}_{4k+2}$ and $a_{i}\in \mathcal{A}$, $i$ varies from $0$ to $n-1.$
	\par For all $i$ and for $j\in\{0,...,2k\},$ define a function $l$ as,
	\begin{eqnarray*} 	
		l(x_{i}^{j}) &=& (j,a_{i})\\
		l(x_{i}^{j'}) &=& (4k+1,a_{0})-l(x_{i}^{j}).
	\end{eqnarray*}
	\par Note that, the label sum of all the vertices of $H_{i}$ is $(2k+1,a_{0})$, which is independent of $i$.
	\par For all $i=0,...,n-1,$ if the degree of vertex $u_{i}$ is $2t_{i}$ for some $t_{i}\geq 1$, then for every vertex $v\in H_{i}$,
	\begin{eqnarray*}
		w_{G\circ H}(v) &=& \sum_{ \substack{v^{*}\in N_{G\circ H}(v), \\ v^{*} \not\in N_{G\circ H[H_{i}]}(v)}}l(v^{*})+ \sum_{v^{**}\in N_{G\circ H[H_{i}]}(v)}l(v^{**})\\
		&=& 2t_{i}(2k+1,a_{0})+2k(4k+1,a_{0}) = (2k+2,a_{0}),
	\end{eqnarray*}
	and, if the degree of vertex $u_{i}$ is $2t_{i}+1$, for some $t_{i}\geq 0,$ then for every vertex $v\in H_{i}$,
	\begin{eqnarray*}
		w_{G\circ H}(v) &=& \sum_{\substack{v^{*}\in N_{G\circ H}(v),\\ v^{*} \not\in N_{G\circ H[H_{i}]}(v)}}l(v^{*})+ \sum_{v^{**}\in N_{G\circ H[H_{i}]}(v)}l(v^{**})\\
		&=& (2t_{i}+1)(2k+1,a_{0})+2k(4k+1,a_{0}) = (1,a_{0}).
	\end{eqnarray*}
	\par Further, the degree of each vertex $u$ of $G$ is congruent to $m$ modulo $(4k+2)$. Then, for every vertex $v$ of $G\times H$,
	%	\begin{equation*}
	%	w_{G\circ H}(v) = \bigl(((4k+2)t+m)(2k+1)+2k(4k+1), a_{0}\bigr)=(2mk+m-2k,a_{0}),
	%	\end{equation*}
	%	and
	\begin{eqnarray*}
		w_{G\times H}(v) &=& \sum_{v^{*}\in N_{G\times H}(v)}l(v^{*}) = 2k((4k+2)k'+m)(4k+1,a_{0}) = (-2mk,a_{0}).
	\end{eqnarray*}
\end{proof}
\par In the following results, we assume $G$ is a graph with vertices $u_{0},...,u_{n-1}$ and $H$ is a balanced-$dmg$ with the vertices $ x^{0},x^{0'},...,x^{2^{k-1}-1},x^{(2^{k-1}-1)'}$, in which $x^{j}$ and $x^{j'}$ are the twin vertices for all $j\in\{0,...,2^{k-1}-1\}$. Moreover, for any $i\in\{0,...,n-1\}$, we choose $H_{i}=\bigl\{x_{i}^{0}, x_{i}^{0'},...,x_{i}^{2^{k-1}-1}, x_{i}^{(2^{k-1}-1)'}\bigr\}$ as the vertex set of $G \circ H$ and $G\times H$, that replaces $u_i$ of $G$ in which $x_{i}^{j}$ and $x_{i}^{j'}$ are the twin vertices. 

%Let {x_{i}^{j}$ and $x_{i}^{j'}$ are the vertices of $G \circ H$ and $G\times H$ that replace $u_i$, for all $i$ in $G$. Now for $j = 0,1,...,2^{k-1}-1$, $x_{i}^{j}$ and $x_{i}^{j'}$ form the twin vertices of $i$th copy of $H$ respectively, $i = 0,1,..,n-1.$  Further $\mathcal{B}$ represent an abelian group of finite order and $b_{0}$ represents the identity element of $\mathcal{B}$.

\begin{lemma}\label{lem1}
	Let $G$ be a graph on $n$ vertices and $\Gamma$ be an abelian group with $|\Gamma|= 2^{k}n$, where $k\geq 2$ such that $\Gamma\cong \mathbb{Z}_{2^s}\times \mathcal{A}$ for $1\leq s \leq k-1$, $\mathcal{A}$ an abelian group with $|\mathcal{A}|=2^{k-s}n$. Let $H$ be a balanced-$dmg$ on $2^{k}$ vertices. Then,
	\begin{enumerate}
		\item[\textnormal{(i)}] $G\circ H$ is $\Gamma$-distance magic.
		\item[\textnormal{(ii)}]If there exists a constant $m\in \mathbb{N}$ such that $deg_G(u) \equiv m\mod2^{s}$ for all $u\in V(G)$, then $G\times H$ is $\Gamma$-distance magic.	
	\end{enumerate}
	
\end{lemma}
\begin{proof}
	\par Using the isomorphism $\phi:\Gamma \rightarrow \mathbb{Z}_{2^s}\times \mathcal{A}$, we identify $g\in \Gamma$ with its image $\phi(g) = (z,a_{i})$, where $z\in \mathbb{Z}_{2^s}$ and $a_{i}\in \mathcal{A}$, $i$ varies from $0$ to $2^{k-s}n-1.$
	\par For $i \in \{0,...,n-1\}$ and $\alpha \in\{0,...,2^{s-1}-1\}$, define a function $l$ as
	\begin{eqnarray*}
		l(x_{i}^{j}) &=& \bigl(\alpha,a_{j~(\textnormal{mod}~ 2^{k-s})+2^{k-s}i}\bigr), \textnormal{~where}~ \alpha 2^{k-s}\leq j\leq (\alpha+1)2^{k-s}-1,\\
		l(x_{i}^{j'}) &=& (2^{s}-1,a_{0})-l(x_{i}^{j}).
	\end{eqnarray*} Now for each $ i=0,...,n-1$, the label sum of all the vertices of $H_{i}$ is,
	\begin{equation*}
	2^{k-1}\bigl(2^{s}-1, a_{0}\bigr)=\bigl(-2^{k-1}, a_{0}\bigr)=(0, a_{0}),
	\end{equation*}which is the identity element of $\mathbb{Z}_{2^{s}} \times \mathcal{A}$ and label sum is independent of $i$.
	\par Note that, the degree of each vertex of $H$ is $2r$.
	For all $i=0,...,n-1,$ the vertex $v\in H_{i}$ has weight,
	\begin{eqnarray*}
		w_{G\circ H}(v) &=& \sum_{\substack{v^{*}\in N_{G\circ H}(v),\\ v^{*}\not\in N_{G\circ H[H_{i}]}(v)}}l(v^{*})+ \sum_{v^{**}\in N_{G\circ H[H_{i}]}(v)}l(v^{**})\\
		&=& deg_{G}(u_{i})(0,a_{0})+r(2^{s}-1, a_{0}) = (-r,a_{0}).
	\end{eqnarray*}
	\par Moreover, if the degree of each vertex $u$ of $G$ is congruent to $m$ modulo $2^{s}$ then, for every $v$ of $G\times H$,
	\begin{eqnarray*}
		w_{G\times H}(v) &=& \sum_{v^{*}\in N_{G\times H}(v)}l(v^{*}) = r(2^sk'+m)(2^s-1,a_{0}) = (-mr,a_{0}).
	\end{eqnarray*}
\end{proof}

\begin{lemma}\label{lem2}
	Let $G$ be a graph on $n$ vertices and $\Gamma$ be an abelian group with $|\Gamma|= 2^{k}n$, such that $\Gamma\cong \mathbb{Z}_{2^s}\times \mathcal{A}$, where $2 \leq k \leq s$ and $\mathcal{A}$ is an abelian group with $|\mathcal{A}|=2^{k-s}n$.  Let $H$ be a balanced-$dmg$ on $2^{k}$ vertices.
	\begin{enumerate}
		\item[\textnormal{(i)}] If there exists a constant $m\in \mathbb{N}$ such that $deg_G(u) \equiv m\mod2^{s-1}$ for all $u\in V(G)$, then $G\circ H$ is $\Gamma$-distance magic.
		\item[\textnormal{(ii)}] If there exists a constant $m\in \mathbb{N}$ such that $deg_G(u) \equiv m\mod2^{s}$ for all $u\in V(G)$, then $G\times H$ is $\Gamma$-distance magic.
	\end{enumerate}
	
\end{lemma}
\begin{proof}
	\par Using the isomorphism $\phi:\Gamma \rightarrow \mathbb{Z}_{2^s}\times \mathcal{A}$, we identify $g\in \Gamma$ with its image $\phi(g) = (z,a_{i})$, where $z\in \mathbb{Z}_{2^s}$ and $a_{i}\in \mathcal{A}$, $i$ varies from $0$ to $2^{k-s}n-1.$
	\par	Consider the function $l$,
	\begin{eqnarray*}
		l(x_{i}^{j}) &=&\bigl((2^{k-1}i+j)\mod 2^{s-1},a_{\lfloor 2^{k-s}i\rfloor}\bigr),\\
		l(x_{i}^{j'}) &=& \bigl(2^{s}-1,a_{0}\bigr)-l(x_{i}^{j}),
	\end{eqnarray*}
	where $i\in\{0,...,n-1\}$ and $j\in\{0,...,2^{k-1}-1\}$.		
	For each $i=0,...,n-1$, the label sum of all the vertices of $H_{i}$ is $(-2^{k-1},a_{0})$, which is independent of $i$. Recall that the degree of any vertex of $H$ is $2r$. Since the degree of any vertex $u$ of $G$ is congruent to $m$ modulo $2^{s-1}$,
	for all $i=0,...,n-1,$ the vertex $v\in H_{i}$ has weight,
	\begin{eqnarray*}
		w_{G\circ H}(v) &=& \sum_{\substack{v^{*}\in N_{G\circ H}(v),\\ v^{*}\not\in N_{G\circ H[H_{i}]}(v)}}l(v^{*})+ \sum_{v^{**}\in N_{G\circ H[H_{i}]}(v)}l(v^{**})\\
		&=& (2^{s-1}k'+m)(-2^{k-1},a_{0})+r(2^s-1,a_{0})= (-r-2^{k-1}m, a_{0}).
	\end{eqnarray*}
	\par On the other hand, if the degree of any vertex $u$ of $G$ is congruent to $m$ modulo $2^{s}$ then for every $v$ of $G\times H$,
	\begin{eqnarray*}
		w_{G\times H}(v) &=& \sum_{v^{*}\in N_{G\times H}(v)}l(v^{*}) = r(2^sk'+m)(2^s-1,a_{0}) = (-mr,a_{0}).
	\end{eqnarray*}
\end{proof}
\begin{theorem}
	Let $G$ be a graph on $n$ vertices and $\Gamma$ be an abelian group with $|\Gamma|= 2^kn$, where $k\geq 2$, $n=2^s(2t+1)$, for some non-negative integers $s$ and $t$. Let $H$ be a balanced-$dmg$ on $2^k$ vertices.
	\begin{enumerate}
		\item[\textnormal{(i)}] If there exists a constant $m\in \mathbb{N}$ such that $deg_G(u) \equiv m\mod2^{k+s-1}$ for all $u\in V(G)$, then $G\circ H$ is $\Gamma$-distance magic.
		\item[\textnormal{(ii)}] If there exists a constant $m\in \mathbb{N}$ such that $deg_G(u) \equiv m\mod2^{k+s}$ for all $u\in V(G)$, then $G\times H$ is $\Gamma$-distance magic.	\end{enumerate}
\end{theorem}
\begin{proof}
	By the fundamental theorem of finite abelian groups,  $\Gamma\cong \mathbb{Z}_{2^{n_{0}}}\times \mathbb{Z}_{p_{1}^{n_{1}}} \times ...\times \mathbb{Z}_{p_{r}^{n_{r}}}$ where $2^kn = 2^{n_{0}}\prod_{i=1}^{r}p_{i}^{n_{i}}$, $p_{i}$'s not necessarily distinct primes and $n_{0}>0$.
	\par Note that if any vertex $u$ of $G$ is such that $deg_G(u)\equiv m \mod 2^{k+s}$, then there exist unique integers $m_{i}$'s such that $deg_G(u)\equiv m_{i} \mod 2^{n_{0}}$, where $n_{0} \in\{1,...,k+s-1\}$.
	\par For each $\Gamma$ isomorphic to $\mathbb{Z}_{2^{n_{0}}}\times \mathcal{A}$ with $1\leq n_{0}\leq k-1$, the result follows from Lemma \ref{lem1} and for each $\Gamma$ isomorphic to $\mathbb{Z}_{2^{n_{0}}}\times \mathcal{A}$ with $k\leq n_{0}\leq k+s$, the result follows from Lemma \ref{lem2}, where $\mathcal{A}$ is an abelian group  with $|\mathcal{A}|=\frac{n}{2^{n_{0}-k}}$.
\end{proof}
\begin{theorem}\label{thmnew1}
	Let $G$ be a graph on $n$ vertices and $\Gamma$ be an Abelian group with $|\Gamma|= 2^{k}n$, where $k\geq 2$. If all the vertices of $G$ are of even degree and $H$ is a balanced-$dmg$ on $2^{k}$ vertices, then $G\circ H$ is a $\Gamma$-distance magic graph.
\end{theorem}
\begin{proof}
	If $\Gamma$ is isomorphic to $\mathbb{Z}_{2^p} \times \mathcal{A}$ for $p\in\{1,...,k-1\}$, then the result follows from Lemma \ref{lem1}. Now, suppose that $\Gamma$ is isomorphic to $\mathbb{Z}_{2^k} \times \mathcal{A}$, where $|\mathcal{A}|=n$.\par Using the isomorphism $\phi:\Gamma \rightarrow \mathbb{Z}_{2^k}\times \mathcal{A}$, we identify $g\in \Gamma$ with its image $\phi(g) = (z,a_{i})$, where $z\in \mathbb{Z}_{2^k}$ and $a_{i}\in \mathcal{A}$, $i$ varies from $0$ to $n-1.$
	\par For all $i\in\{0,...,n-1\}$ and $j\in\{0,...,2^{k-1}-1\}$, define $l$ on $G\circ H$ as,
	\begin{eqnarray*}
		l(x_{i}^{j}) &=& (2j,a_{i}),\\
		l(x_{i}^{j'}) &=& (2^{k}-1,a_{0})-l(x_{i}^{j}).
	\end{eqnarray*}
	\par Note that, the label sum of all the vertices of $H_{i}$ is,
	\begin{equation*}
	(2^{k-1}(2^{k}-1)\mod 2^k, a_{0})=(-2^{k-1}, a_{0}),
	\end{equation*}which is independent of $i$.
	\par Since for any $u_{i}$ of $G$, $deg_{G}(u_{i})=2t_{i}$ with $t_{i}\geq 1$ and for any $x$ of $H$, $deg_{H}(x)=2r$, then the degree of $v$ in $G\circ H$ is $2(2^kt_{i}+r)$. Now, the weight of any vertex $v\in H_{i}$ is,
	\begin{eqnarray*}
		w_{G\circ H}(v) &=& \sum_{\substack{v^{*}\in N_{G\circ H}(v)\\ v^{*}\not\in N_{G\circ H[H_{i}]}(v)}}l(v^{*})+ \sum_{v^{**}\in N_{G\circ H[H_{i}]}(v)}l(v^{**})\\
		&=& 2t_{i}(-2^{k-1}, a_{0})+r(2^k-1,a_{0})= (-r,a_{0}).
	\end{eqnarray*}
\end{proof}
\begin{corollary}
	Let $t$ be an odd integer. Let $G\cong K_{m_{1},m_{2},...,m_{t}}$ be a complete $t$-partite graph with, $m=\sum_{i=1}^{t}m_{i}$ and either all $m_{i}$'s are even or all $m_{i}$'s are odd. If $\Gamma$ is an abelian group with $|\Gamma| = 2^km$, then $G\circ H$ is a $\Gamma$-distance magic graph, where $H$ is a balanced-dmg on $2^k$ vertices.\qed
\end{corollary}
\par For an abelian group $\Gamma$, the following result discusses the $\Gamma$-distance magic labeling of $K_{m,n}\circ H$, where $m$ and $n$ are of different parity and $H$ is a balanced-dmg on $2^k$ vertices.
\begin{theorem}
	Let $K_{m,n}$ be a complete bipartite graph with $m$ even and $n$ odd and let $\Gamma$ be an abelian group with $2^k(m+n)$ elements, where $k\geq 2$.  If $H$ is a $2r$-regular balanced-$dmg$ on $2^k$ vertices and $r$ is odd, then $K_{m,n}\circ H$ is $\Gamma$-distance magic.
\end{theorem}
\begin{proof}
	If $\Gamma$ is isomorphic to $\mathbb{Z}_{2^p} \times \mathcal{A}$ with $p\in\{1,...,k-1\}$, then the assertion follows from Lemma \ref{lem1}. Suppose that $\Gamma$ is isomorphic to $\mathbb{Z}_{2^k} \times \mathcal{A}$, where $|\mathcal{A}|=m+n$. \par Let $G\cong K_{m,n}$ have the partition sets $X = \{{u}_0,..., {u}_{m-1}\}$ and $Y = \{ {v}_0,..., {v}_{n-1}\}$. Then for each $i\in\{0,...,m-1\}$, let $X_{i}=\{{x}_{i}^0,x_{i}^{0'},...,$$ {x}_{i}^{2^{k-1}-1}$ $,x_{i}^{(2^{k-1}-1)'}\} $ be the vertex set of $G\circ H$, that replace the vertex $u_{i}$ of $G$. Similarly, for each $j\in\{0,...,n-1\}$, let  $Y_{j}=\{{y}_{j}^0,y_{j}^{0'},..., {y}_{j}^{2^{k-1}-1},y_{j}^{(2^{k-1}-1)'}\} $ be the vertex set of $G\circ H$, that replace the vertex $v_{j}$ of $G$.
	%Let $A'=\cup_{i=0}^{m-1}\cup_{j=0}^{2^{k-1}-1}\{x_{i}^{j},x_{i}^{j'}\}$ and $B'=\cup_{s=0}^{n-1}\cup_{j=0}^{2^{k-1}-1}\{y_{s}^{j},y_{s}^{j'}\}$.
	\par Now, let the vertex set of $G\circ H$ be $X'\cup Y'$, where $\displaystyle X'=\bigcup_{i=0}^{m-1}X_{i}$ and $Y'=\displaystyle \bigcup_{j=0}^{n-1}Y_{j}$.
	\par Using the isomorphism $\phi:\Gamma \rightarrow \mathbb{Z}_{2^k}\times \mathcal{A}$, we identify $g\in \Gamma$ with its image $\phi(g) = (z,a_{i})$, where $z\in \mathbb{Z}_{2^k}$ and $a_{i}\in \mathcal{A}$, $i$ varies from $0$ to $m+n-1.$
	\par For each $q\in\{0,...,2^{k-1}-1\}$, define $l$ on $X'$ as,
	\begin{eqnarray*}
		l(x_{i}^{q}) &=& ((2^{k-1}+1)q,a_{i}), \textnormal{~and} \\
		l(x_{i}^{q'}) &=& (2^{k-1}-1,a_{0})-l(x_{i}^{q}), \textnormal{~~~for ~all}~ i=0,...,m-1.
	\end{eqnarray*}
	\par  Again, for each $q\in\{0,...,2^{k-1}-1\}$, define $l$ on $Y'$ as,
	\begin{eqnarray*}
		l(y_{j}^{q}) &=& (2q,a_{m+j}),  \textnormal{~and} \\
		l(y_{j}^{q'}) &=& (2^{k}-1,a_{0})-l(y_{j}^{q}), \textnormal{~~~for ~all}~ j=0,...,n-1.
	\end{eqnarray*}
	\par Then for all $i\in\{ 0,...,m-1\}$, the label sum of all vertices of $X_{i}$ is,
	\begin{equation*}
	(2^{k-1}(2^{k-1}-1),a_{0}) = (2^{k-1},a_{0}).
	\end{equation*}
	\par Similarly, for all $j\in\{0,...,n-1\}$, the label sum of all vertices of $Y_{j}$ is,
	\begin{equation*}
	(2^{k-1}(2^{k}-1),a_{0}) = (2^{k-1},a_{0}).
	\end{equation*} \par Further, since $H$ is $2(2t+1)$-regular graph, for every $x$ of $X_{i}$,
	
	%\substack{i=0 \\ i\neq 4}
	
	\begin{eqnarray*}
		w_{G\circ H}(x) &=& \sum_{ \substack{x^{*}\in N_{G\circ H}(x) \textnormal{,} \\ x^{*}\not\in N_{G\circ H[X_{i}]}(x)}}l(x^{*})+ \sum_{x^{**}\in N_{G\circ H[X_{i}]}(x)}l(x^{**})\\
		&=& n(2^{k-1},a_{0})+ (2t+1)(2^{k-1}-1,a_{0}) = \bigl(-(2t+1),a_{0}\bigr),
	\end{eqnarray*}
	and for every $y$ of $Y_{j}$,
	\begin{eqnarray*}
		w_{G\circ H}(y) &=& \sum_{\substack{y^{*}\in N_{G\circ H}(y) \textnormal{,} \\ y^*\not\in N_{G\circ H[Y_{j}]}(y)}}l(y^{*})+ \sum_{y^{**}\in N_{G\circ H[Y_{j}]}(y)}l(y^{**})\\
		&=& m(2^{k-1},a_{0})+ (2t+1)(2^{k}-1,a_{0}) = \bigl(-(2t+1),a_{0}\bigr),
	\end{eqnarray*}
	which completes the proof.
\end{proof}
\section{Conclusion}
\noindent In this paper, we obtain few necessary conditions for a graph to be group distance magic and characterize the group distance magic labeling of a tree, few subclasses of bi-regular graphs and the lexicographic and direct product of a non-regular graph with a balanced distance magic graph.
%\bibliographystyle{model1a-num-names.bst}
%\bibliography{bibiliography_list}

\end{document}